\documentclass[12pt]{amsart}
\usepackage{amsmath,amsthm}
\usepackage{amssymb}
\usepackage{graphicx}
\usepackage[all]{xypic}
\SelectTips{cm}{12}
\UseTips{}
\usepackage{euscript}
\usepackage{amsfonts}
\usepackage{latexsym}
\usepackage{amssymb}
\usepackage{euscript}
\usepackage{epsf}
\usepackage{graphicx}
\theoremstyle{plain}
\swapnumbers
    \newtheorem{thm}{Theorem}[section]

    \newtheorem{prop}[thm]{Proposition}
    \newtheorem{lemma}[thm]{Lemma}

    \newtheorem{subsec}[thm]{}
\theoremstyle{definition}
    \newtheorem{defn}[thm]{Definition}
    \newtheorem{example}[thm]{Example}

\theoremstyle{remark}
        \newtheorem{remark}[thm]{Remark}
	\newtheorem{ack}[thm]{Acknowledgements}
        
        \newtheorem{summary}[thm]{Summary}
\newenvironment{myeq}[1][]
{\stepcounter{thm}\begin{equation}\tag{\thethm}{#1}}
{\end{equation}}

\newcommand{\mydiagram}[2][]
{\stepcounter{thm}\begin{equation}
     \tag{\thethm}{#1}\vcenter{\xymatrix{#2}}\end{equation}}
\newenvironment{mysubsection}[2][]
{\begin{subsec}\begin{upshape}\begin{bfseries}{#2.}
\end{bfseries}{#1}}
{\end{upshape}\end{subsec}}
\newenvironment{mysubsect}[2][]
{\begin{subsec}\begin{upshape}\begin{bfseries}{{#2}\vsn.}
\end{bfseries}{#1}}
{\end{upshape}\end{subsec}}
\newcommand{\w}[2][ ]{\ \ensuremath{#2}{#1}\ }
\newcommand{\ww}[2][ ]{\ \ensuremath{#2}{#1}}
\newcommand{\wh}{\ -- \ }

\newcommand{\wb}[2][ ]{\ (\ensuremath{#2}){#1}\ }
\newcommand{\wwb}[1]{\ (\ensuremath{#1})-}
\newcommand{\wref}[2][ ]{\ \eqref{#2}{#1}\ }
\newcommand{\hsp}{\hspace{10 mm}}

\newcommand{\hsm}{\hspace{2 mm}}

\newcommand{\vsm}{\vspace{2 mm}}
\newcommand{\vsn}{\vspace{1 mm}}
\newcommand{\hra}{\hookrightarrow}

\newcommand{\rest}[1]{\lvert_{#1}}

\newcommand{\clos}{\operatorname{cl}}
\newcommand{\pclos}{\operatorname{pcl}}

\newcommand{\cspace}{configuration space}

\newcommand{\Conf}{\operatorname{Conf}}
\newcommand{\dd}{\operatorname{d}}
\newcommand{\ddp}{\dd\!}

\newcommand{\dhh}{\operatorname{dh}}
\newcommand{\Emb}[2]{\operatorname{Emb}^{#1}({#2})}

\newcommand{\Euc}[1]{\operatorname{Euc}^{#1}}

\newcommand{\Id}{\operatorname{Id}}
\newcommand{\Image}{\operatorname{Im}}

\newcommand{\Line}{\operatorname{Line}}
\newcommand{\open}{\operatorname{op}}

\newcommand{\rank}{\operatorname{Rank}}

\newcommand{\wspace}{work space}
\newcommand{\hphi}{\hat{\phi}}
\newcommand{\tpsi}{\tilde{\psi}}
\newcommand{\be}{\mathbf{e}}
\newcommand{\ve}{\vec{\be}}
\newcommand{\vel}{\vec{\ell}}

\newcommand{\bte}{\mathbf{t}}
\newcommand{\vt}{\vec{\bte}}

\newcommand{\bv}{\mathbf{v}}
\newcommand{\vv}{\vec{\bv}}
\newcommand{\bx}{\mathbf{x}}
\newcommand{\xe}{\bx_{e}}
\newcommand{\xs}{\bx_{\star}}

\newcommand{\bw}{\mathbf{w}}
\newcommand{\vw}{\vec{\bw}}

\newcommand{\by}{\mathbf{y}}
\newcommand{\bz}{\mathbf{z}}
\newcommand{\bze}{\mathbf{0}}
\newcommand{\vz}{\vec{\bze}}
\newcommand{\R}{{\mathbb R}}
\newcommand{\RR}[1]{\R^{#1}}
\newcommand{\C}{{\EuScript C}}
\newcommand{\Cs}{\C_{\ast}}
\newcommand{\hC}{\widehat{\C}}
\newcommand{\hCs}{\hC_{\ast}}
\newcommand{\hCps}{\hC'_{\ast}}
\newcommand{\CG}{\C(\Gamma)}
\newcommand{\CGp}{\C(\Gp)}

\newcommand{\CsG}{\Cs(\Gamma)}
\newcommand{\CsGp}{\Cs(\Gp)}

\newcommand{\FG}{F^{\Gamma}}

\newcommand{\T}{{\EuScript T}}
\newcommand{\TG}{\T_{\Gamma}}
\newcommand{\Vp}{\Vc\,'}
\newcommand{\Vpp}{\Vc\,''}
\newcommand{\hV}{\hat{\Vc}}
\newcommand{\hVp}{\hV'}
\newcommand{\Vk}{\Vj{k}}
\newcommand{\Vn}{\Vj{n}}
\newcommand{\hVnp}{\Vpj{n+1}}

\newcommand{\Pc}{{\mathcal P}}

\newcommand{\Vc}{{\mathcal V}}

\newcommand{\Wc}{{\mathcal W}}

\newcommand{\Gp}{\Gamma'}
\newcommand{\tGp}{\tilde{\Gamma}'}
\newcommand{\Gpp}{\Gamma''}

\newcommand{\Gnl}[1]{\Gamma_{\open}^{#1}}
\newcommand{\Gkl}{\Gnl{k}}
\newcommand{\Gkprc}[1]{\Gamma_{\pclos}^{#1}}
\newcommand{\Gkc}[1]{\Gamma_{\clos}^{#1}}
\newcommand{\Gkpc}{\Gkc{k+1}}

\newcommand{\up}[1]{\sp{({#1})}}

\newcommand{\lj}[1]{\ell\up{#1}}

\newcommand{\nj}[1]{n\up{#1}}

\newcommand{\Vj}[1]{\Vc\up{#1}}
\newcommand{\Vpj}[1]{\hat{\Vc}\up{#1}}
\newcommand{\vj}[1]{\bv\up{#1}}

\newcommand{\xj}[1]{\bx\up{#1}}
\newcommand{\xjk}[2]{\bx\up{#1}\sb{#2}}

\begin{document}
\title{Generic singular configurations of linkages}
\author{David Blanc}
\address{Dept.\ of Mathematics, U. Haifa, 31905 Haifa, Israel}
\email{blanc@math.haifa.ac.il}
\author{Nir Shvalb}
\address{Dept.\ of Industrial Engineering, Ariel Univ.\ Center, 47000,
Ariel, Israel}
\email{nirsh@ariel.ac.il}
\date{\today}
\begin{abstract}
We study the topological and differentiable singularities of the
\cspace\ \w{\CG} of a mechanical linkage $\Gamma$ in \w[,]{\RR{d}}
defining an inductive sufficient condition to determine when a
configuration is singular. We show that this condition holds for
generic singularities, provide a mechanical interpretation, and
give an example of a type of mechanism for which this criterion
identifies all singularities.
\end{abstract}
\maketitle

%
%
\section{Introduction}
\label{cint}
The mathematical theory of robotics is based on the notion of a
mechanism consisting of links, joints, and rigid platforms.
The \emph{mechanism type} is a simplicial (or polyhedral) complex
\w[,]{\TG} where the parts of dimension $\geq 2$ \ correspond to the
platforms, and the complementary one-dimensional graph corresponds to
the links (=edges) and joints (=vertices).  The \emph{linkage} (or
mechanism) $\Gamma$ itself is determined by assigning fixed lengths to
each of the links of \w[.]{\TG} See \cite{MerlP,SeliG,TsaiR} and
\cite{FarbT} for surveys of the mechanical and topological aspects,
respectively. 

\begin{mysubsection}{Configuration spaces}\label{scspace}
Here we concentrate on the most prevalent type of mechanism
\w[:]{\TG} namely, a finite $1$-dimensional simplicial
complex (undirected graph), with $N$ vertices and $k$ edges. 
Note that a rigid platform is completely specified by listing the
lengths of all its diagonals (i.e., the distance between any two 
vertices), so we need not list the platforms explicitly. Our results
actually hold also for the case when some links of $\Gamma$ are
\emph{prismatic} (or telescopic) \wh i.e., have variable length \wh
but for simplicity we deal here with the fixed-length case only.

A length-preserving embedding of the vertices of the linkage $\Gamma$
in a fixed ambient Euclidean space \w{\RR{d}} is called a
\emph{configuration} of $\Gamma$. In applications, $d$ is most
commonly $2$ or $3$. The set of all such embeddings, with the natural
topology (and differentiable structure), is called the \cspace\ of
$\Gamma$, denoted by \w[.]{\CG} Such \cspace s have been studied
intensively, with the hope of extracting useful mechanical information
from their topological or geometric properties. Much of the
mathematical literature has been devoted to the special case when
$\Gamma$ is a closed chain (polygon): see, e.g., 
\cite{FTYuzvT,HausT,HKnutC,KMillM,KMillS,MTrinG}. However, the
general case has also been treated (cf.\
\cite{HolcM,KamiHA,KTsuC,KMillU,OHaraM,SSBlaCA,SSBlaCP}).
\end{mysubsection}

\begin{mysubsection}{Singularities}\label{ssing}
There are two main types of singularities which arise in robotics.
The \emph{kinematic} singularities of a mechanism, which appear as
singularities of work and actuation maps defined on \w{\CG} (\S
\ref{dwork}), have obvious mechanical interpretations, and have been
studied intensively (see, e.g., \cite{GAngS}, \cite[\S 6.2]{MerlP},
and \cite{ZFBenS}). On the other hand, the \emph{topological} or
differentiable singularities of the \cspace\ \w{\CG} itself have not
received much attention in the literature since  \cite{HuntK},
aside from some special examples (see, e.g., \cite{FarbT,KMillS} and
\cite{ZBGossC}). 

For any linkage $\Gamma$, the \cspace\ \w{\CG} is the zero set of a
smooth function \w{\lambda:\RR{Nd}\to\RR{k}} (see \S \ref{dcspace} below),
so that \w{\CG} is typically a smooth manifold (when \w{\vz\in\RR{d}}
is a regular value of $\lambda$), and even if not, ``most'' points of
\w{\CG} are smooth, since a simple \emph{necessary} condition for a
point $\Vc$ in \w{\CG} to be singular is that
\w[.]{\rank(\ddp\lambda_{\Vc})<k} Thus we are in the common situation
where it is relatively straightforward to identify configurations
which are \emph{possibly} singular, but not so easy to pinpoint when
this is in fact so\vsm . 

Our goal in this paper is threefold:

\begin{enumerate}\renewcommand{\labelenumi}{(\alph{enumi})}
\item To provide a straightforward inductive description of a
  \emph{sufficient} condition for a configuration $\Vc$ to be
  differentiably singular (in fact, this will imply that $\Vc$ is even
  a topological singularity) \wh see Proposition \ref{pgensing} and
  Theorem \ref{tgensing}.
\item To show that this condition applies generically (that is, to all
  but a positive-codimension subset of the singular locus $\Sigma$)
  \wh see Remarks \ref{rgensin} and \ref{rgensnt}.
\item To obtain a mechanical interpretation for all singularities
in the \cspace\ of a linkage $\Gamma$ as a tangential
conjunction of two kinematic singularities of type I (cf.\
\cite{GAngS}) for complementary sub-mechanisms of $\Gamma$ \wh see
Remark \ref{rmechanism}.
\end{enumerate}

The third goal is completely achieved only in the plane (for
\w[),]{d=2} since the model we use for \cspace s is not completely
realistic for rigid rods in \w[.]{\RR{3}} See Remark \ref{rccs} below
for an explanation of the difficulties involved.
\end{mysubsection}

\begin{remark}\label{ralgvar}
Since the function \w{f:\RR{Nd}\to\RR{k}} defining the \cspace\ is
a quadratic polynomial (cf.\ \S \ref{dcspace}), \w{\CG} is actually
a real algebraic variety. Thus any topological or differentiable
singularity $\Vc$ is in particular an algebraic singularity (cf.\
\cite[Ch.\ II, \S 1.4]{ShafB1}). Somewhat more surprisingly, every
real algebraic variety is a union of components of the \cspace\
of some planar linkage \wb{d=2} \wh see \cite{KMillU,KingP,JSteiC}. 
Thus our results here appear to be statements about any real algebraic
variety. 

However, the point we wish to make here is not that the cone
singularities are the most common ones in algebraic varieties; it
is rather the mechanical interpretation of the generic
singularities, and the mechanical underpinnings of the inductive
process described in Section \ref{ciccspaces}.

In fact, while the topological, differentiable, and geometric structures
on \cspace s of linkages can be used to study their mechanics (cf.\
\cite{KMillS,KTezuS}), the algebraic structure usually plays no role (but
see \cite{CapoG}).
\end{remark}

\begin{mysubsection}[\label{sorg}]{Organization}
In Section \ref{cbcspaces} we briefly review some of the basic
notions used in this paper. In Section \ref{clecspaces}, various
concepts of local equivalences of \cspace s are defined; these
help to simplify the study of singular points. In Section
\ref{cpbc} we explain the role played by pullbacks of \cspace s. This
is applied in Section \ref{ciccspaces} to provide an inductive
construction, which is used both to describe the sufficient condition
mentioned in \S \ref{ssing}(b), and to show that they are indeed
singular points. An example is studied in detail in Section \ref{ctriang}. 
\end{mysubsection}

\begin{ack}
We wish to thank the referee for his or her comments.
\end{ack}

%
%
\section{Background on configuration spaces}
\label{cbcspaces}

We first recall some general background material on the
construction and basic properties of \cspace s. This also serves
to fix notation, which is not always consistent in the literature.

\begin{defn}\label{dcspace}
Consider an abstract graph \w{\TG} with vertices $V$ and edges
\w[.]{E\subseteq V^{2}} A \emph{linkage} (or \emph{mechanism})
$\Gamma$ of type \w{\TG} is determined by a function \w{\ell:E\to\RR{}_{+}}
specifying the length \w{\ell_{i}} of each edge \w{e_{i}} 
in \w{E=\{e_{i}=(u_{i},v_{i})\}_{i=1}^{k}} (subject to the triangle
inequality as needed).  We write
\w{\vel^{2}:=(\ell_{1}^{2},\dotsc,\ell_{k}^{2})\in\RR{E}} 
for the vector of squared lengths.

The set of all embeddings of $V$ in an ambient Euclidean space
\w{\RR{d}} is an open metric subspace of \w[,]{(\RR{d})^{V}} denoted
by \w[.]{\Emb{d}{\TG}} We have a \emph{squared length map}
\w{\lambda:\Emb{d}{\TG}\to\RR{E}} with 
\w[,]{\lambda(u_{i},v_{i}):=\|\varphi(u_{i})-\varphi(v_{i})\|^{2}} and 
the \emph{\cspace} of the linkage \w{\Gamma=(\TG,\ell)} is the
metric subspace \w{\CG:=\lambda^{-1}(\vel^{2})} of \w[.]{\Emb{d}{\TG}}
A point \w{\Vc\in\CG} is called a \emph{configuration} of $\Gamma$. 
Note that $\lambda$ is an algebraic function of \w{\Vc\in\RR{dN}}
(which is why the lengths were squared), so \w{\CG} is a real
algebraic variety. 
\end{defn}

\begin{remark}\label{rsubmer}
By \cite[I, Theorem 3.2]{HirsD}, we know that \w{\CG} is a smooth
manifold if \w{\vel^{2}} is a regular value of $\lambda$: that is, if its
differential \w{\ddp\lambda_{\Vc}} is of maximal rank for every
\w{\Vc\in\Emb{d}{\TG}} with \w[.]{\lambda(\Vc)=\vel^{2}} 

However, for some mechanism types \w[,]{\TG} this condition may not be
generic: there exist  mechanism types \w{\TG} and an open set $U$ in
\w{\RR{dN}} consisting of non-regular values of \w[.]{\FG} This means
that for each \w[,]{\vel^{2}_{0}\in U} the \cspace\
\w{\C(\Gamma_{\vel^{2}_{0}}):=\lambda^{-1}(\vel^{2}_{0})} has at least 
one configuration \w{\Vc\in\C(\Gamma_{\vel^{2}_{0}})} such that
$\lambda$ not a submersion at $\Vc$. See \cite{SSBlaCP} for an example.
\end{remark}

\begin{mysubsection}{Isometries of \cspace s}
\label{srestr}
The group \w{\Euc{d}} of isometries of the Euclidean space
\w{\RR{d}} acts on the space \w[.]{\CG} When $\Gamma$ has a rigid
``base platform'' $P$ of dimension $\geq d-1$, \  
this action is free. In this case we can work with the ``restricted
\cspace'' \w[,]{\CG/\Euc{d}} and the quotient map has a continuous
section (equivalent to choosing a fixed location in \w{\RR{d}} for
$P$). See \S \ref{sppm} for an example of such a $\Gamma$.

In general, certain configurations (e.g., those contained in a proper
linear subspace $W$ of \w[)]{\RR{d}} may be fixed by certain
transformations (those fixing $W$), so the action of \w{\Euc{d}} is not free.
\end{mysubsection}

\begin{defn}\label{dfixlink}
Choose a fixed vertex \w{\xs} of $\Gamma$ as its \emph{base-point}:
the action of the translation subgroup \w{T\cong\RR{d}} of \w{\Euc{d}}
on \w{\xs} \emph{is} free, so its action on \w{\CG} is free, too,
and we call the quotient space \w{\CsG:=\CG/T} the \emph{pointed \cspace} 
for $\Gamma$. Thus \w[,]{\CG\cong\CsG\times\RR{d}} and a pointed
configuration (i.e., an element of \w[)]{\CsG} is simply an ordinary
configuration expressed in terms of a coordinate frame for \w{\RR{d}}
with the origin at \w[.]{\xs}  

If we also choose a fixed link $\vv$ in $\Gamma$ starting at \w[,]{\xs} 
we obtain a smooth map \w{p:\CsG\to S^{d-1}} which assigns to a
configuration $\Vc$ the direction of $\vv$. The fiber \w{\hCs(\Gamma)}
of $p$ at \w{\ve_{1}\in S^{d-1}} will be called the \emph{reduced \cspace} 
of $\Gamma$. Note that the bundle \w{\CsG\to S^{d-1}} is locally
trivial. 
\end{defn}

\begin{defn}\label{dwork}
A mechanism $\Gamma$ may be equipped with a special point \w{\xe} \wh in
engineering terms this is the ``end-effector'' of $\Gamma$, whose
manipulation is the goal of the mechanism. We think of
\w{\Delta:=\{\xe\}} as a sub-mechanism of $\Gamma$ (more generally, we
could choose any rigid sub-mechanism).
Assuming that the base-point \w{\xs} of $\Gamma$ is not \w[,]{\xe}
the inclusion \w{j:\Delta\hra\Gamma} induces a map of \cspace s
\w[,]{j^{\ast}:\CsG\to\C(\Delta)} whose image $\Wc$ is  
called the \emph{work space} of the mechanism. The \emph{work map}
\w{\psi:\CsG\to\Wc} of $\Gamma$ is the factorization of
\w{j^{\ast}} through $\Wc$ (which is not always a smooth manifold).  
\end{defn}

\begin{example}\label{eglwork}
Now consider a closed $5$-chain  \w[,]{\Gkc{5}} as in Figure
\ref{ffivebar}, with end-effector \w[.]{\xe=\xj{2}} Here 
the direction of \w{\vv:=\xj{4}-\xj{0}} is fixed.

\begin{figure}[htb]
\epsfysize=3.5cm \leavevmode \epsffile{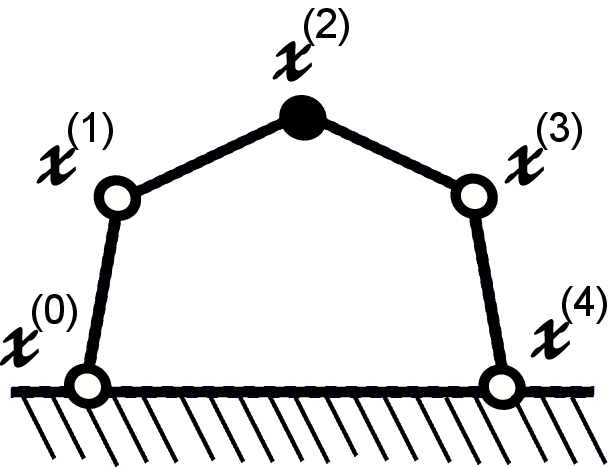}
\caption{Closed $5$-chain \ $\Gkc{5}$} \label{ffivebar}
\end{figure}

 The work space of each of the two open sub-chains of \w{\Gkc{5}}
 starting at \w{\xj{0}} and ending at \w{\xj{2}} is a closed annulus.
 Therefore, $\Wc$ is the intersection of these two annuli (see Figure
\ref{fworkfive}), i.e. a curvilinear polygon in \w[,]{\RR{2}}
whose combinatorial type depends on the lengths of the links.
\end{example}

\begin{figure}[htb]
\epsfysize=3.5cm
\leavevmode \epsffile{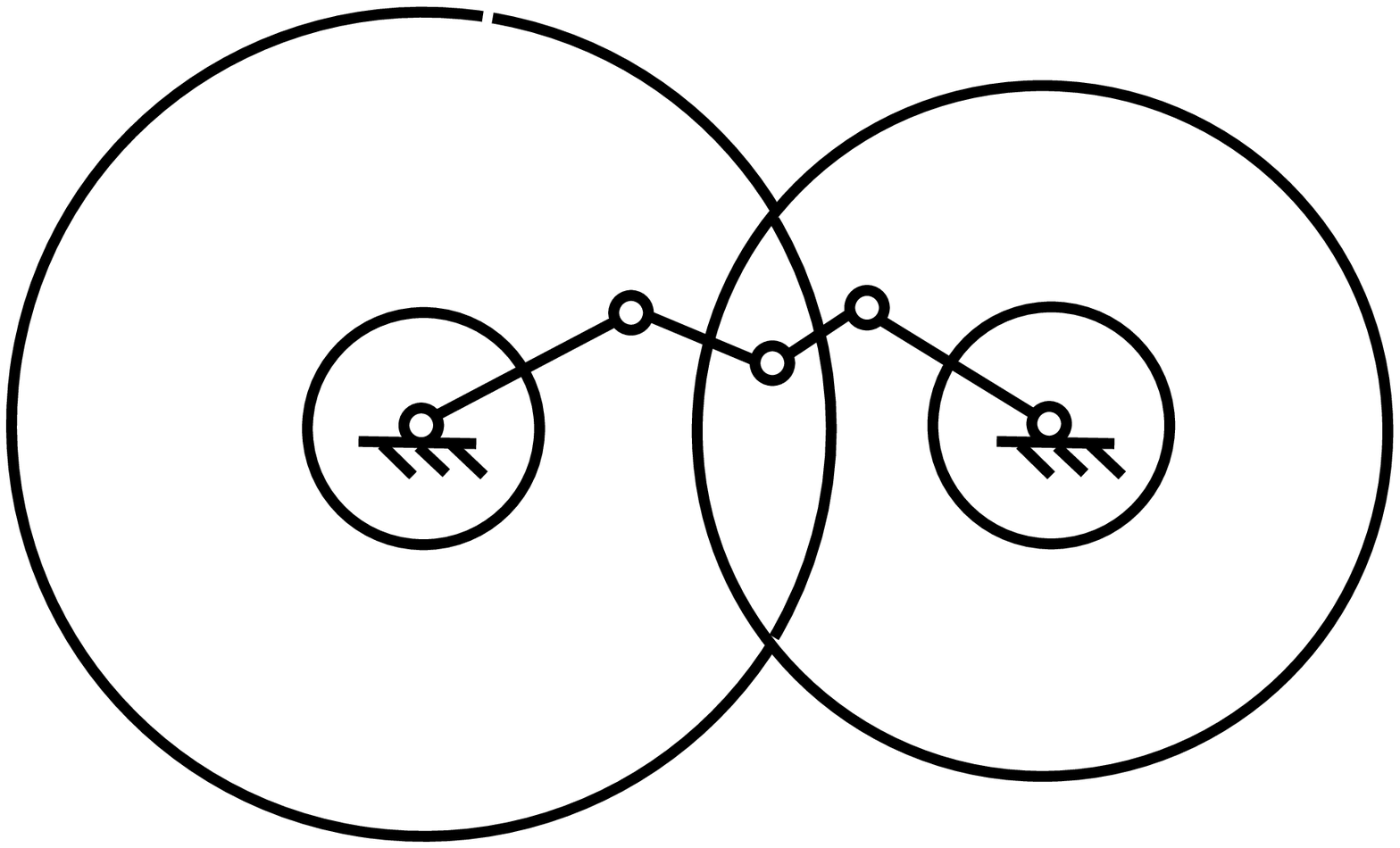}
\caption{The lens-shaped work space $\Wc$ for \ $\Gkc{5}$}
\label{fworkfive}
\end{figure}

\begin{remark}\label{rccs}
The \cspace s studied in this paper are mathematical models, which take
into account only the locations of the vertices of $\Gamma$,
disregarding possible  intersections of the edges. In the plane, there
is some justification for this, since we can allow one link to slide
over another. This is why this model is commonly used
(cf. \cite{FarbT,KMillM}; but see \cite{CDRoteS}). However, in
\w{\RR{3}} the model is not very realistic, since it disregards the
fact that rigid rods cannot pass through each other.

Thus a proper treatment of configurations in \w{\RR{3}} must cut our
``naive'' version of \w{\Emb{d}{\TG}} (and thus \w{\CG} and
\w[)]{\CsG} along the subspace of configurations which are not
embeddings of the full graph \w[.]{\TG} The precise description of
such a ``realistic'' \cspace\ \w{\Conf(\TG)} is quite complicated,
even at the combinatorial level, which is why we work here
with \w[,]{\Emb{d}{\TG}} \w[,]{\CG} and \w{\CsG} as defined in \S
\ref{dcspace}-\ref{dfixlink}.  Note, however, that \w{\CG} has a dense open 
subspace \w{U(\Gamma)} consisting of embeddings of the full graph
(including its edges), which may be identified with a dense open subset
of \w[.]{\Conf(\TG)} We observe that even such a model \w{\Conf(\TG)}
is not completely realistic, in that it disregards the thickness of
the rigid rods.

Unfortunately, the generic singularities we identify here are not in 
\w[.]{U(\Gamma)} Nevertheless, in some cases at least, our method of
replacing one singular configuration by another (see Section
\ref{clecspaces} below) allows us to replace the generic 
singularity in \w{\CG\setminus U(\Gamma)} with a configuration in
\w[,]{U(\Gamma')} for a suitable linkage \w[.]{\Gamma'} See Section
\ref{ctriang} for an example of this phenomenon (which also occurs in
the $3$-dimensional version of the linkage described there).
\end{remark}

%
%
\section{Local equivalences of \cspace s}
\label{clecspaces}

Let $\Gamma$ and \w{\Gp} be two linkages.  We would like to think
of points in the respective \cspace s as being equivalent if they
are both smooth, or both have ``similar'' singularities. Since these
concepts are local, we make the following: 

\begin{defn}\label{dloce}
Two configurations $\Vc$ in \w{\CG} and \w{\Vp} in \w{\CGp}
are:

\begin{enumerate}\renewcommand{\labelenumi}{(\alph{enumi})}
\item \emph{locally equivalent} if there are neighborhoods
$U$ of $\Vc$ in \w{\CsG}  and \w{U'} of \w{\Vp} in
\w[,]{\Cs(\Gp)}  and a homeomorphism \w{f:U\to U'} with
\w[.]{f(\Vc)=\Vp}
\item \emph{locally product-equivalent} if there are neighborhoods
$W$ of $\Vc$ in \w{\CsG}  and \w{W'} of \w{\Vp} in
\w{\Cs(\Gp)} equipped with homeomorphisms \w{W\cong U\times\RR{k}}
(taking $\Vc$ to \w[)]{(\Vc_{0},\bx)} and \w{W'\cong U'\times\RR{m}}
(taking \w{\Vp} to \w[),]{(\Vp_{0},\by)} as well as a homeomorphism
\w{f:U\to U'} with \w[.]{f(\Vc_{0})=\Vp_{0}}
\end{enumerate}
See \cite{KMillU} for other formulations of this and similar notions.
\end{defn}

Evidently, any two smooth configurations in any two \cspace s are
locally product-equivalent\vsm.

In the next section we decompose our \cspace s into simpler factors
(locally), gluing them along appropriate work maps. The singularities
of the \cspace s translate into work singularities on the factors, so
we need an analogous notion of work maps being locally equivalent (at
smooth configurations), or locally equivalent up to a Euclidean
factor:

\begin{defn}\label{dlocew}
If \w{i:\Delta\hra\Gamma} and \w{i':\Delta\hra\Gamma'} are inclusions
of a common rigid sub-mechanism $\Delta$ (usually a single point) in two
distinct linkages, and \w[,]{\Vc\in\CsG} \w{\Vp\in\Cs(\Gp)} are two
smooth configurations, we say that \w{i^{\ast}} and \w{(i')^{\ast}} are 

\begin{enumerate}\renewcommand{\labelenumi}{(\alph{enumi})}
\item \emph{work-equivalent} at \w{(\Vc,\Vp)} if there are
  neighborhoods $U$ of $\Vc$, \w{U'} of \w[,]{\Vp} and $W$ of
  \w[,]{i^{\ast}(\Vc)=(i')^{\ast}(\Vp)} and a diffeomorphism $f$
  making the following diagram commute: 
\mydiagram[\label{eqlocew}]{
 & \ar@_{(->}[dl] \hspace*{1mm} U
\ar[dddr]_{i^{\ast}\rest{U}}\ar[rr]^{f}_{\cong} &&
U'\hspace*{1mm}\ar@^{(->}[rd] \ar[dddl]^{(i')^{\ast}\rest{U'}} & \\
\CsG \ar[d]_{i^{\ast}} & && & \Cs(\Gp) \ar[d]^{(i')^{\ast}} \\
\Cs(\Delta) & && & \Cs(\Delta') \\
 && \ar@_{(->}[llu]\hspace*{1mm} W \hspace*{1mm}\ar@^{(->}[rru] &&
}
\item \emph{$S$-equivalent} at \w{(\Vc,\Vp)} if there are
neighborhoods \w{W\cong U\times\RR{k}} of $\Vc$ and \w{W'\cong
U'\times\RR{m}} of \w{\Vp} and a homeomorphism \w{f:U\to U'} as in
\S \ref{dloce}(b) above, such \w{i^{\ast}} factors through the
projection \w{\pi:W\to U} and \w{(i')^{\ast}} factors through
\w{\pi':W'\to U'} in such a way that the diagram analogous to
\wref{eqlocew} commutes.
\end{enumerate}
\end{defn}

An important example of these notions is provided by the following
simple mechanism:

\begin{defn}\label{dopen}
An \emph{open $k$-chain} is a linkage \w[,]{\Gkl}
where \w{\TG} is a connected linear graph with \w{k+1} vertices (where
all but the endpoints \w{\xj{0}} and \w{\xj{k}} are of valency $2$),
with lengths  \w[.]{(\ell_{1},\dotsc,\ell_{k})} See Figure \ref{fopenchain}
below. It is natural to choose the base-point \w{\xs:=\xj{0}} (fixed at the
origin, say) to define the pointed \cspace\ \w[,]{\Cs(\Gkl)} and
\w{\xe:=\xj{k}} as end-effector.  

The resulting workspace $\Wc$ is \w[,]{S^{d-1}\times[m,M]} for fixed 
\w[,]{0<m<M} where \w{m=\min\{|\sum_{i=1}^{k}\pm\ell_{i}|\}} and
\w{M=\sum_{i=1}^{k}\ell_{i}} are respectively the minimal and
maximal possible distances of \w{\xe} from \w[.]{\xs} The spherical
(or polar) coordinate \w{\theta\in S^{d-1}} is the direction of the
vector \w[.]{\vv=\xe-\xs} 

A \emph{closed \wwb{k+1}chain} is a linkage \w[,]{\Gkc{k+1}}
where \w{\TG} is a cycle with \w{k+1} vertices (of valency $2$), having
lengths  \w{\ell_{1}=|\xj{1}-\xj{0}|,\ell_{2}=|\xj{2}-\xj{1}|,\dotsc,
\ell_{k+1}=|\xj{0}-\xj{k}|}  (see Figure \ref{ffivebar}). 

A \emph{prismatic closed \wwb{k+1}chain} \w{\Gkprc{k+1}} has the same 
\w[,]{\TG} with lengths \w{(\ell_{1},\dotsc,\ell_{k})} as for \w[,]{\Gkc{k+1}}
but with the last link prismatic \wh that is, the length 
\w{\ell=|\xj{0}-\xj{k}|} varies in the range \w[.]{m\leq\ell\leq M} 
\end{defn}

\begin{lemma}[\protect\cite{GottT}]\label{lwork}
The work map $\psi$ of an open chain is a submersion, unless \w{\Vc} is
aligned (that is, all links have a common direction vector $\vw$ in 
\w{\RR{d}} at $\Vc$). In this case the \ $(d-1)$-dimensional subspace
\w{\Image(\ddp\psi)_{\Vc}} is orthogonal to $\vw$.
\end{lemma}

Clearly the \cspace s of an open $k$-chain and the corresponding 
prismatic closed \wwb{k+1}chain are isomorphic. However, 
the following result will be useful in understanding the work map
singularities of an open chain, by allowing us to disregard its
\wwb{d-1}dimensional non-singular direction.

\begin{prop}\label{pzero}
If \w{\Gkl} is an open $k$-chain with links 
\w[,]{(\ell_{1},\dotsc,\ell_{k})} then the \emph{pointed} \cspace 
\w{\Cs(\Gkl)} is $S$-equivalent at any configuration $\Vc$ to the \emph{reduced}
\cspace\ \w{\hCs(\Gkprc{k+1})} of a closed prismatic \wwb{k+1}chain.
\end{prop}

\begin{figure}[htb]
\epsfysize=3.5cm
\leavevmode \epsffile{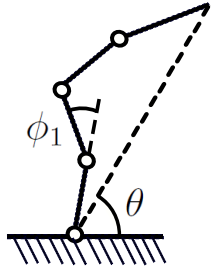}
\caption{Coordinates for the  open chain}
\label{fopenchain}
\end{figure}

\begin{proof}
We may choose \w{(\theta,\phi_{1},\dotsc,\phi_{k-1})\in(S^{d-1})^{k}} as
local coordinates for the smooth \cspace\ \w{\Cs(\Gkl)} near $\Vc$, where
\w{\phi_{i}} is the spherical angle between the vectors
\w{\xj{i-1}\xj{i}} and \w{\xj{i}\xj{i+1}} (see Figure
\ref{fopenchain}), and $\theta$ is as in \S \ref{dopen} 
(for \w[).]{\vv\neq\vz} 

Thus in a coordinate neighborhood 
\w{U\cong\RR{k(d-1)}} of $\Vc$ the work map
\w{i^{\ast}:U\to\RR{d-1}\times[m,M]} factors as
\w[,]{(\pi_{\theta},\rho)} where
\w{\pi_{\theta}(\theta,\phi_{1},\dotsc,\phi_{k-1})=\theta} is the
projection, and
\w[.]{\rho(\theta,\phi_{1},\dotsc,\phi_{k-1})=\|\xj{k}-\xj{0}\|\in[m,M]}

Now for each \w[,]{\ell\in[m,M]} the fiber \w{\rho^{-1}(\ell)}
is diffeomorphic to the \cspace\ \w{\Cs(\Gkpc)} of a closed chain
having \w{k+1} links of lengths
\w[.]{(\ell_{1},\dotsc,\ell_{k},\ell)} As in \S \ref{dfixlink},
we have \w[,]{\Cs(\Gkpc)\cong S^{d-1}\times\hCs(\Gkpc)} so
\w{\Cs(\Gkpc)} is locally product-equivalent to
\w[,]{\hCs(\Gkpc)}and in fact \w{\Cs(\Gkpc)} is $S$-equivalent to
\w{\hCs(\Gkpc)} with respect to \w[.]{\Delta=\{\xj{k}\}} As $\ell$
varies, we obtain the mechanism \w[.]{\Gkprc{k+1}}

If \w{\vv:=\xe-\xs} vanishes at $\Vc$, but $\Vc$ is not aligned, then the
work map $\psi$ is a submersion at $\Vc$, and the same holds for 
\w[,]{\Cs(\Gkprc{k+1})} so they are $S$-equivalent. If \w{\vv=\bz} at
$\Vc$ and $\Vc$ is aligned, choose the coordinate $\theta$ be the
direction of the alignment vector $\vw$. 
\end{proof}

\begin{mysubsect}[\label{sdwm}]{Decomposing the work map}

Consider an arbitrary mechanism $\Gamma$ with base point \w{\xs} and
work map \w{\psi:\CsG\to\RR{d}} for the end-effector \w[.]{\xe} Note
that \w{\CsG} is locally diffeomorphic to the product
\w{S^{d-1}\times\hCs(\Gamma)} (\S \ref{dfixlink}), since the bundle
\w{\hCs(\Gamma)\hra\CsG\to S^{d-1}} (for \w[)]{\vv:=\xe-\xs\in S^{d-1}} 
is locally trivial (assuming $\vv$ does not vanish). If we choose
local spherical coordinates \w{S^{d-1}\times\RR{}_{+}} for the work
space \w[,]{\Wc\subseteq\C(\Delta)\subseteq\RR{d}} the work map
\w{\psi:\CsG\to\Wc\subseteq S^{d-1}\times\RR{}_{+}} may be written
locally in the form
\begin{myeq}\label{eqpsi}
\psi=\Id_{S^{d-1}}\times\tpsi:S^{d-1}\times\hCs~\to~
S^{d-1}\times\RR{}_{+}
\end{myeq}
\noindent for some smooth function \w{\tpsi:\hCs\to\RR{}_{+}} (which
is the work function for the associated reduced \cspace).
Note that the derivative of the work function $\psi$ may thus be
written in the form:
\begin{myeq}\label{eqdpsi}
(\ddp\psi)_{(\vv,\hV)}~=~\left(\begin{array}{ll} I_{d-1} & \bze \\ \bze
  & (\nabla\tpsi)_{\hV}\end{array}\right).
\end{myeq}
\noindent which shows that \w{\ddp\psi} has rank $d$ or \w[.]{d-1}
\end{mysubsect}

\begin{prop}\label{pgensin}
If \w{\Vc=(\hV,\Vc')\in\CsG} is a smooth configuration for a
mechanism $\Gamma$ with work function
\w{\psi=\Id_{S^{d-1}}\times\tpsi} as in \wref[,]{eqpsi} with
\w[,]{\xe\neq\xs} and $\hV$ is a non-degenerate singular point of
$\tpsi$, then \w{\CsG} is $S$-equivalent at $\Vc$ to an aligned
configuration of an open $n$-chain for some \w[.]{n\geq 1}
\end{prop}

\begin{proof}
By the Morse Lemma (cf.\ \cite[Theorem 2.16]{MatsM}) we may choose
local coordinates \w{\vt=(t_{1},\dotsc,t_{k-d+1})} for
\w{\hCs(\Gamma)} near $\hV$ (where \w[),]{k=\dim\CsG} so that
$\tpsi$ has the form
\begin{myeq}\label{eqmorse}
\tpsi(\vt)=a_{0}+\sum_{i=1}^{j}t_{i}^{2}-\sum_{i=j+1}^{k-d+1}t_{i}^{2}~.
\end{myeq}
\noindent On the other hand, by Proposition \ref{pzero} the
\cspace\ \w{\Cs(\Gnl{n})} for an open $n$-chain at any
configuration \w{\Vn} is $S$-equivalent to the reduced \cspace\
\w{\hCs(\Gkprc{n+1})} at some configuration \w[,]{\hVnp} where 
\w{\Gkprc{n+1}} is a prismatic closed \wwb{n+1}chain. The reduced work
map 
$$
\hphi:\hCs(\Gkprc{n+1})\to\gamma\subseteq\Wc\subseteq\RR{d}
$$
\noindent assigns to each \w{\hV\in\hCs(\Gkprc{n+1})} the length of the
variable link (with \w[,]{\gamma\cong[m,M]} the segment of possible lengths).

As shown in \cite[Theorem 5.4]{MTrinG}, $\hphi$ is a Morse function,
having (non-degenerate) singular points precisely at the aligned
configurations \w{\hVnp} of the closed chain \w[.]{\Gkprc{n+1}}
Although Milgram and Trinkle do not calculate the index of $\hphi$ at
\w[,]{\hVnp} their computation of the Hessian of $\hphi$ in
\cite[Key Example, p.~255]{MTrinG}, combined with Farber's proof of
\cite[Lemma 1.4]{FarbT} for the planar case, show that this index is
equal to \w[,]{n-k} where $k$ is the number of forward-pointing links
in the configuration \w[.]{\hVnp} Thus by the Morse Lemma again we may
choose an aligned configuration \w{\hVnp} and local coordinates in
\w{\hCs(\Gkprc{n+1})} around it so that $\hphi$ too has the form
\wref[,]{eqmorse} and thus \w{\CsG} is $S$-equivalent at $\Vc$ to
\w{\hCs(\Gkprc{n+1})} at \w[.]{\hVnp} By Proposition \ref{pzero} it is then
readily seen to be $S$-equivalent to \w{\Cs(\Gnl{n})} at the
corresponding aligned open-chain configuration \w[.]{\Vn}
\end{proof}

%
%
\section{Pullbacks of \cspace s}\label{cpbc}

We now describe a procedure for viewing the \cspace\ of an arbitrary linkage
$\Gamma$ as a pullback, obtained by decomposing $\Gamma$ into two
simpler sub-mechanisms. The basic idea is a familiar one \wh see,
e.g., \cite{MTrinG}.

\begin{mysubsection}{Pullbacks}\label{spb}
Let \w{\Gkl} denote an open chain which is a sub-mechanism of $\Gamma$
(cf.\ \S \ref{dopen}), and let \w{\Gp} denote the mechanism obtained from
$\Gamma$ by omitting the $k$ links of \w{\Gkl} (and all vertices but
\w{\xj{0}} and \w[).]{\xj{k}} For simplicity we choose \w{\xs:=\xj{0}}
as the common base-point of $\Gamma$, \w[,]{\Gkl} and \w[,]{\Gp} and
\w{\xe:=\xj{k}} as the common end-effector of \w{\Gkl} and \w[.]{\Gp} 
See Figure \ref{fdecompose}.

\begin{figure}[htb]
\epsfysize=3.5cm \leavevmode \epsffile{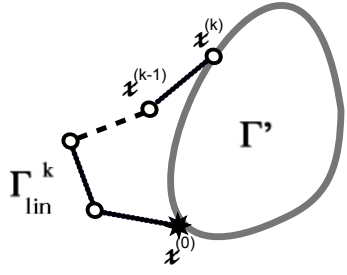}
\caption{Decomposing $\Gamma$ into two sub-mechanisms}
\label{fdecompose}
\end{figure}

The \wspace\ of both mechanisms \w{\Gp} and \w{\Gkl} (i.e., the set of possible
locations for \w[)]{\xe} is contained in \w[,]{\RR{d}}
and we have work maps \w{\psi:\Cs(\Gp)\to\RR{d}} and 
\w{\phi:\Cs(\Gkl)\to\RR{d}} which associate to each configuration the
location of \w[.]{\xe}

Note that the pointed \cspace\ \w{\Cs(\Gkl)} is a manifold (diffeomorphic to
\w[)]{(S^{d-1})^{k}} with a natural embedding \w[,]{i:\Cs(\Gkl)\hra\RR{kd}}
and similarly \w{j:\Cs(\Gp)\to\RR{M}} for a suitable Euclidean space
\w[.]{\RR{M}}  This can be done, for example, by using the position
coordinates in \w{\RR{d}} for every vertex in $\Gamma$.

Let \w{X:=\Cs(\Gkl)\times\RR{M}} and \w[,]{Y:=\RR{d}\times\RR{M}} and define
\w{h:X\to Y} to be the product map \w{\phi\times\Id_{\RR{M}}} and
\w{g:\Cs(\Gp)\to Y} to be \w[,]{(\psi,j)} so that $g$ is an embedding of
\w{\Cs(\Gp)} as a submanifold in $Y$. Since we have a pullback square:
\mydiagram[\label{eqpb}]{
\CsG \ar[d]\ar[r] & \Cs(\Gkl) \ar[d]^{\phi}\\
\Cs(\Gp)\ar[r]^{\psi} & \Wc\subseteq\RR{d}~,
}
\noindent \w{\CsG} may be identified with the preimage of the subspace
\w{\Cs(\Gp)\subseteq Y} under $h$\vsm.

Let \w{\Vp\in\Cs(\Gp)} and \w{\Vk\in\Cs(\Gkl)} be matching configurations with
\w[,]{\psi(\Vp)=\phi(\Vk)} and let \w{\bx\in X} be the
configuration \w[,]{(\Vk,j(\Vp))} so that \w[:]{h(\bx)=g(\Vp)}
\mydiagram[\label{eqpullback}]{
\hspace*{3.5mm}\bx\in X\ar@<2.7ex>[d]_{h} & = &
\Cs(\Gkl)\ar@<-2.8ex>[d]_{\phi}\ni\Vk & \times &
\RR{M}\ar@<-3.7ex>[d]_{\Id}\ni j(\Vp)\\
h(\bx)\in Y           & = & \hspace*{8mm}\RR{d}\ni\phi(\Vk)  & \times &
\RR{M}\ni j(\Vp) \\
\hspace*{0.5mm}\Vp\in\Cs(\Gp)
\ar@^{(->}[u]<-2.7ex>^{g} \ar[urr]^{\psi}\ar@_{(->}[urrrr]_{j} & & & &
}
\noindent We want to know if the point \w{\Vc\in\CsG} defined by \w{(\Vp,\Vk)}
is singular. By \cite[I, Theorem 3.3]{HirsD}, $\Vc$ is smooth if
\w{h\pitchfork\Cs(\Gp)} \wh i.e., $h$ is locally transverse to \w{\Cs(\Gp)} at the
points \w{\bx\in X} and \w[,]{\Vp\in\Cs(\Gp)} which means that
\w[.]{\Image\dhh_{\bx}~+~T_{\Vp}(\Cs(\Gp))~=~T_{\Vp}(Y)~=~\RR{d}\times\RR{M}}

Since \w{\Id_{\RR{M}}} is onto, this is equivalent to:
\begin{myeq}\label{etransv}
\Image(\ddp\phi)_{\Vk}~+~\Image(\ddp\psi)_{\Vp}~=~\RR{d}
\end{myeq}
\end{mysubsection}

\begin{mysubsection}{Generic singularities in pullbacks}
\label{sgensin}
Clearly, the failure of \wref{etransv} is a \emph{necessary}
condition for \w{\Vc=(\Vp,\Vk)} to be singular in \w[.]{\CsG} Note
that if \wref{etransv} does not hold, then neither
\w{(\ddp\phi)_{\Vk_{n}}} nor \w{(\ddp\psi)_{\Vp_{n}}} is onto
\w[.]{\RR{d}} By Lemma \ref{lwork}, the first implies that the configuration
\w{\Vk_{n}} for the open chain \w{\Gnl{k}} must be aligned, while the
second implies that \w{(\ddp\psi)_{\Vp_{n}}} is of rank $<d$.
\end{mysubsection}

\begin{defn}\label{dgensing}
Given a pullback diagram as in \wref[,]{eqpb} a configuration
\w{(\Vp,\Vk)\in\CsG\subseteq\Cs(\Gp)\times\Cs(\Gkl)} will be called
\emph{generically non-transverse} if \w{\hVp} is a non-degenerate
singular point of $\tpsi$, and \w[.]{\xj{0}\neq\xj{k}}
\end{defn}

\begin{remark}\label{rgensin}
Note that since \w{\tpsi:\hCps\to\RR{}_{+}} is an algebraic function,
generically it will be a Morse function, so any singular point
\w{\hVp} is non-degenerate. Likewise, in the moduli space
\w{\Lambda=\RR{k}_{+}} for open $k$-chains, the
subspace of moduli $\lambda$ for which \w{\Gkl} has no aligned
configurations with \w{\xj{0}=\xj{k}} is Zariski open in
$\Lambda$. Thus among the potentially singular configurations of
\w{\CsG} (i.e., those for which \wref{etransv} fails), the generically
non-transverse ones are indeed generic.
\end{remark}

\begin{prop}\label{pgensing}
Given a pullback diagram \wref[,]{eqpb} any generically
non-transverse configuration \w{(\Vp,\Vk)} is the product of a
Euclidean space with a cone on a homogeneous quadratic hypersurface,
so in particular it is a topological singularity of \w[.]{\CsG}
\end{prop}

\begin{proof}
Since \w{\hVp} is a non-degenerate singular point of $\tpsi$, by
Proposition \ref{pgensin} the work map
\w{\psi:\CsGp\to\Wc\subseteq\RR{d}} is work-equivalent to the
work map $\eta$ of an open chain \w{\Gnl{n}} at some aligned
configuration \w[.]{\Vn} Thus the pullback diagram \wref{eqpb} may be
replaced by one of the form
\mydiagram[\label{eqlinpb}]{
\CsG \ar[d]\ar[r] & \Cs(\Gkl) \ar[d]^{\phi}\\
\Cs(\Gnl{n})\ar[r]^{\eta} & \Wc\subseteq\RR{d}~,
}
\noindent so that $\C$ itself is $S$-equivalent at
\w{(\Vk,\Vn)} to the \cspace\ of a closed chain with \w{(n+k)}
links at an aligned configuration (since $\phi$ and $\eta$ were
non-transverse). This is known to be the cone on a homogeneous
quadratic hypersurface, by \cite[Theorem 1.6]{FarbT} and
\cite[Theorem 2.6]{KMillS}, so it is topologically singular.
\end{proof}

%
%
\section{Inductive construction of \cspace s}\label{ciccspaces}

We now define an inductive process for studying the local behavior of
a configuration $\Vc$ of a linkage $\Gamma$. This consists of
successively discarding open chains of $\Gamma$ while preserving the
local structure. 

\begin{mysubsection}{The inductive procedure}\label{sip}
We saw in \S \ref{spb} how removing an open chain sub-mechanism from
$\Gamma$ allows one to describe the \cspace\ \w{\CsG} as a
pullback of two \cspace s \w{\Cs(\Gkl)} and \w[,]{\Cs(\Gp)} where
the first is completely understood, and the second is simpler than
the original \w[.]{\CsG}

This idea may now be applied again to \w[:]{\Cs(\Gp)} by
repeatedly discarding (or adding) open chain sub-mechanisms, we
construct a sequence of pullbacks
\mydiagram[\label{eqpullb}]{
\C(\Gamma_{n+1}) \ar[d]\ar[r] & \C(\Lambda_{n}) \ar[d]^{\phi_{n}}\\
\C(\Gamma_{n})\ar[r]^{\psi_{n}} & \RR{d}~,
}
\noindent for \w[,]{1\leq n< M} where each \w{\Gamma_{n-1}} is a
sub-mechanism of \w[,]{\Gamma_{n}} with \w[,]{\Gamma=\Gamma_{M}} and
\w{\Lambda_{n}} is an open chain in \w{\RR{d}} (so
\w{\C(\Lambda_{n})} is a product of \ww{(d-1)}-spheres).
The maps \w{\psi_{n}} and \w{\phi_{n}}are work
maps for the common endpoint of \w{\Gamma_{n}} and \w[.]{\Lambda_{n}}

Each configuration $\Vc$ for $\Gamma$ determines a sequence
of pairs \w{\Vp_{n+1}=(\Vp_{n},\Vk_{n})} in \w[,]{\C(\Gamma_{n+1})} as in
\wref[,]{eqpullb} where \w{\Vk_{n}} is necessarily a smooth point of
\w[.]{\C(\Lambda_{n})} Evidently, if \w{\Vp_{n}} is a smooth point of
\w[,]{\C(\Gamma_{n})} \w{\Vp_{n+1}} will be, too, if \wref{etransv}
holds.
\end{mysubsection}

\begin{remark}\label{rip}
Note that there is usually more than one way to decompose a given
linkage $\Gamma$ as in \S \ref{spb}, so the full inductive process
described above is actually encoded by an (inverted) rooted tree,
with varying degrees at each node (and the root at $\Gamma$ itself).
Any rooted branch \w{(\Gamma_{n},\Lambda_{n})_{n=k}^{M-1}}
\wb{\Gamma_{M}=\Gamma} of this tree will be called a \emph{decomposition}
of $\Gamma$.

This flexibility can be very useful in applying the inductive
procedure (see \S \ref{egsing} below for an example).
\end{remark}

\begin{mysubsection}{Generic singularities in \ $\CsG$}\label{sgensinc}
Our goal is to use this procedure to study singular configurations of
\w[.]{\CG} Here we start with the simplest case, which is also the
\emph{generic} form of singularities in \cspace s, as we shall see below.

Thus we assume by induction that \w{\Vp_{n}} is a smooth
configuration, but \wref{etransv} \emph{fails}. Our goal is to analyze
this failure in the generic case, and then show that in this case
\w{\Vp_{n+1}} is a singular point. Eventually, we would like to use
this to deduce that the original configuration $\Vc$ is singular, too.

In \S \ref{sgensin}, we saw how to identify positively the generic
singularities appearing in one step in the inductive process of \S \ref{sip},
defined by a pullback diagram \wref[:]{eqpullb} namely, if
\w{\Vp_{n+1}\in\C(\Gamma_{n+1})} is defined by a pair of smooth
configurations \w[,]{(\Vp_{n},\Vk_{n})} but \wref{etransv} fails, then
generically at least, \w{\Vp_{n+1}} is a topological singularity.
However, this does not yet guarantee that the corresponding
configuration $\Vc$ in \w{\CsG} itself is singular (unless
\w[,]{\Gamma=\Gamma_{n+1}} of course).
\end{mysubsection}

\begin{example}\label{egsing}
Let \w{\Gkc{4}} be a planar closed $4$-chain with links
of lengths \w[,]{\lj{1},\lj{2},\lj{3}} and \w[.]{\lj{4}} See Figure
\ref{ffourbar}.

\begin{figure}[htb]
\epsfysize=4.0cm
\leavevmode \epsffile{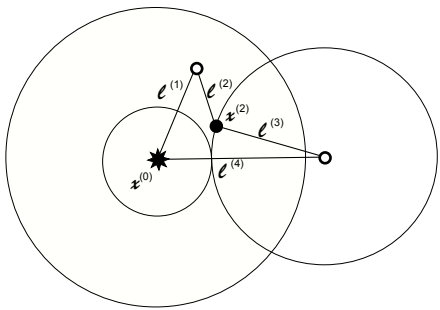}
\caption{Workspace for the point \w{\xj{2}} of a closed $4$-chain}
\label{ffourbar}
\end{figure}

Generically, \w{\hCs(\Gkc{4})} is a smooth $1$-dimensional manifold,
with local parameter given by $\theta$ (the angle between \w{\vj{1}}
and \w[,]{\vj{3}} say). However, if \w[,]{\lj{1}+\lj{3}=\lj{2}+\lj{4}} then
\w{\hCs(\Gkc{4})} has a topological singularity \wh a node \wh at the
aligned configuration $\hV$ where the links \w{\vj{1}} and \w{\vj{3}}
face right, say, and \w{\vj{2}} and \w{\vj{4}} face left (see
\cite[Theorem 1.6]{FarbT}). In fact, if there are no further relations
among \w[,]{\lj{1},\dotsc,\lj{4}} this is the only singularity, and
\w{\hCs(\Gkc{4})} is a figure eight (the one point union of two
circles). We can think of \w{\Gkc{4}} as being decomposed into two
sub-mechanisms \w{\Gp} and \w[,]{\Gpp} each an open $2$-chain: 
\w{\Gp} consisting of \w{\vj{1}} and \w[,]{\vj{2}} and
\w{\Gpp} of \w{\vj{3}} and \w[.]{\vj{4}} Note that
\w[,]{\hV:=(\Vp,\Vpp)} where \w{\Vp} and \w{\Vpp} are both aligned.

In this case we can describe \w{\Cs(\Gkc{4})} explicitly in terms of
the work map \w{\phi:\Cs(\Gkc{4})\to\RR{2}} (for the vertex
\w[),]{\xe:=\xj{2}} which is a four-fold covering map at all points
but $\Vc$: in a punctured neighborhood of $\Vc$, neither 
\w{\Vp} nor \w{\Vpp} can be aligned, and each independently can have
an ``elbow up'' ($+$) or ``elbow down'' ($-$) position, which
together provide the four discrete configurations corresponding to
a single value of $\phi$. In \w[,]{\hCs(\Gkc{4})} taken
together, these give four different branches of the curve
(parameterized by $\theta$) \wh which coincide at $\Vc$. See Figure
\ref{fbranch}.

\begin{figure}[htb]
\epsfysize=3cm \leavevmode \epsffile{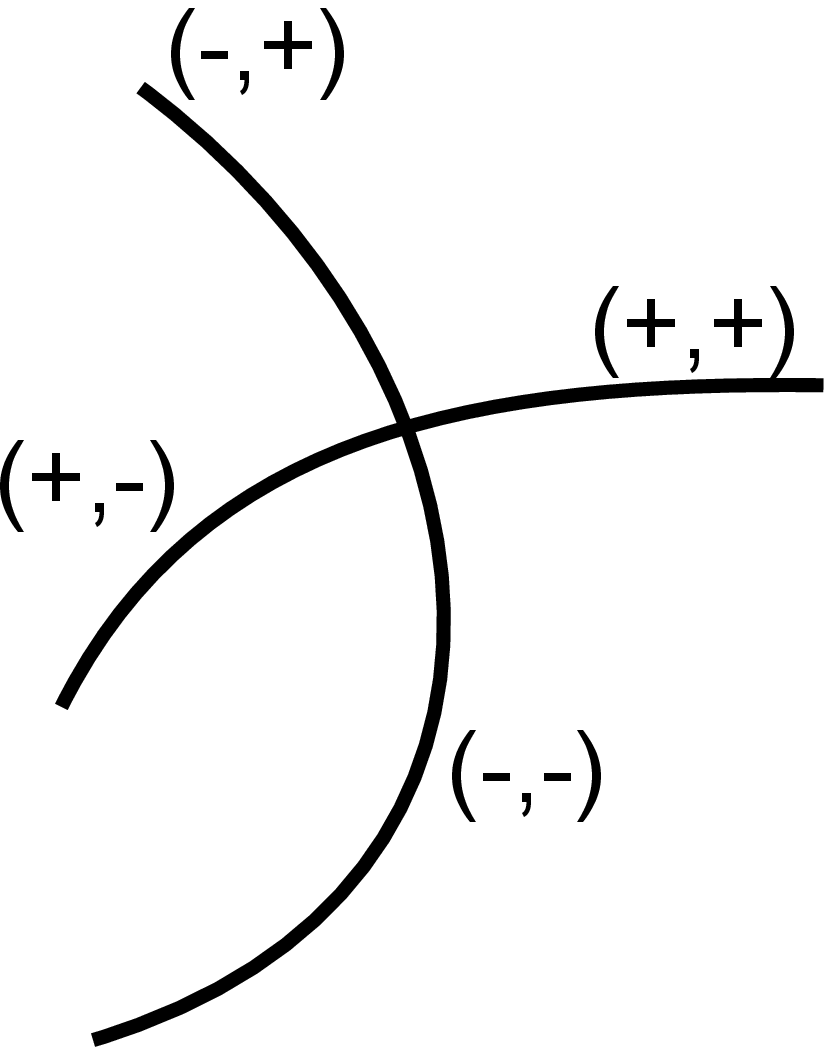} \caption{The
four branches of \ $\hCs(\Gamma_{\clos}^{4})$} \label{fbranch}
\end{figure}

Now assume given a linkage $\Gamma$ in which \w{\Gamma_{2}=\Gkc{4}} as
above (with \w[).]{\lj{1}+\lj{3}=\lj{2}+\lj{4}} Assume that to obtain 
\w{\Gamma_{3}} we add an open $2$-chain \w[,]{\Lambda_{2} \Gnl{2}}
having vertices \w[,]{\xj{0}} \w[,]{\xj{3}} and \w[,]{\xj{4}} with
\w{\|\xj{0}\xj{4}\|=\lj{5}} and \w[.]{\|\xj{3}\xj{4}\|=\lj{6}}  We
therefore now have a rigid triangle \w{\triangle\xj{0}\xj{3}\xj{4}}
(with \w{\xj{4}} in  ``elbow up'' or ``elbow down'' position relative
to the edge \w[).]{\xj{0}\xj{3}} Thus
\w[,]{\Cs(\Gamma_{3})=\Cs(\Gamma_{2})\times\{\pm1\}} and the
singularity at \w{\Vp_{2}:=\hV} is unaffected.  

In the last stage \w{\Gamma=\Gamma_{4}} is obtained by adding another
open $2$-chain \w{\Lambda_{3}:=\Gnl{2}} with one new vertex \w[,]{\xj{5}}
\w{\|\xj{4}\xj{5}\|=\lj{7}} and \w[.]{\|\xj{5}\xj{1}\|=\lj{8}}
We require the configuration \w{\Vj{2}_{2}} of \w{\Lambda_{3}} in which 
\w[,]{\xj{1}} \w[,]{\xj{4}} and \w{\xj{5}} are aligned 
to coincide with the aligned configuration \w{\Vp_{2}=\hV}
of \w{\Gamma_{2}} (and thus \w{\Vp_{3}=(\Vp_{2},+)} of \w[).]{\Gamma_{3}}

The effect of adding \w{\Lambda_{3}} is to prevent the open chain
\w{\Gp=\xj{0}\xj{1}\xj{2}} from ever being in an ``elbow down''
position, thus eliminating two of the four branches of
\w{\hCs(\Gamma_{\clos}^{4})} (see Figure \ref{fbranch}), so 
\w{\Vc:=(\Vp_{3},\Vj{2}_{2})} (which reduces to $\hV$ in
\w[)]{\Cs(\Gamma_{2})} is \emph{not} singular in \w[.]{\CsG} 

To show that this is indeed so, consider an alternative decomposition
of $\Gamma$ (see Remark \ref{rip} above), in which we start with the
closed $5$-chain \w[,]{\Gamma_{1}=\xj{4}\xj{5}\xj{1}\xj{2}\xj{3}} with
base point \w[.]{\xj{3}} See Figure \ref{faltdec}.
Note that \w{\Vp_{1}} corresponding to $\hV$ is non-singular in
\w[.]{\Cs(\Gamma_{1})} When we add the open $2$-chain
\w[,]{\Lambda_{1}=\xj{3}\xj{0}\xj{1}} we see that the configuration
\w{\Vk_{1}} corresponding to $\Vc$ \emph{is} aligned, but since the
work map \w{\phi_{1}:\Cs(\Gamma_{1})\to\RR{2}} determined by the work
point \w{\xj{1}} is a submersion at \w[,]{\Vp_{1}} condition
\wref{etransv} holds at \w[,]{\Vc=(\Vp_{1},\Vk_{1})} so $\Vc$ is smooth.
\end{example}

\begin{figure}[htb]
\epsfysize=5.0cm
\leavevmode \epsffile{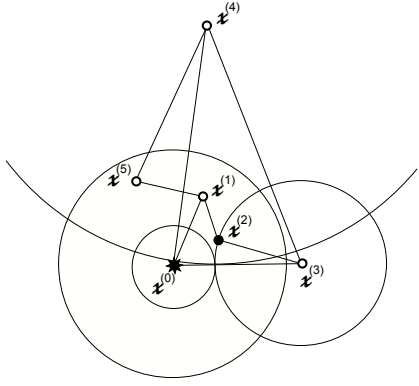}
\caption{An alternative decomposition of $\Gamma$}
\label{faltdec}
\end{figure}

\begin{mysubsection}{Singularities in the inductive process}\label{ssip}
In Example \ref{egsing} we saw that a singularity appearing at one
stage in the inductive process described in \S \ref{sip} can disappear
at a later stage. However, in that case the reason was that the
aligned configuration \w{\Vj{2}_{2}} of \w{\Lambda_{2}=\Gnl{2}}
matched up in \wref{eqpullb} with the aligned configuration
\w{\Vp_{3}} of \w[.]{\Gamma_{3}} 
\end{mysubsection}

\begin{defn}\label{dgennt}
For any linkage $\Gamma$, a configuration \w{\Vc\in\CsG} will be called
\emph{generically non-transversive} if for some decomposition
\w{(\Gamma_{n},\Lambda_{n})_{n=m}^{M-1}} of \w{\Gamma=\Gamma_{M}} 
(see \S \ref{rip}), the pair
\w{(\Vp_{m},\Vk_{m})\in\Cs(\Gamma_{m})\times\Cs(\Lambda_{m})} 
is generically non-transverse in the sense of Definition
\ref{dgensing}, and the open chain configurations
\w{\Vk_{n}\in\Cs(\Lambda_{n})} are not aligned for \w[.]{M>n\geq m}
\end{defn}

\begin{remark}\label{rgensnt}
As noted in Remark \ref{rgensin}, the condition that the original pair
\w{(\Vp_{m},\Vk_{m})} is generically non-transverse is indeed generic,
in the sense that it occurs in a subvariety of
\w{\Cs(\Gamma_{m})\times\Cs(\Lambda_{m})} of positive codimension.
Since the work maps each open chain
\w{\phi_{n}:\Cs(\Lambda_{n})\to\RR{d}} are algebraic for each
\w[,]{n>m} the subvariety of \w{\Cs(\Gamma_{n})\times\Cs(\Lambda_{n})}
consisting of pairs \w{(\Vp_{n},\Vk_{n})} for which \w{\Vp_{n}}
corresponds to \w{\Vp_{n-1}} (and eventually to \w[)]{\Vp_{m}} and
\w{\Vk_{n}} is aligned form a subvariety of positive codimension, so
the condition that $\Vc$ is generically non-transversive in the sense of
Definition \ref{dgennt} is indeed generic among the singular points of
\w[.]{\CsG}
\end{remark}

\begin{thm}\label{tgensing}
For any linkage $\Gamma$, a generically non-transversive
configuration $\Vc$ is a topological singular point of \w{\CsG}
\wh in fact, the product of a cone on a homogeneous quadratic
hypersurface by a Euclidean space.
\end{thm}

\begin{proof}
Let \w{(\Vp_{m},\Vk_{m})} be a generically non-transverse
configuration of
\w[,]{\Cs(\Gamma_{m+1})\subseteq\Cs(\Gamma_{m})\times\Cs(\Lambda_{m})}
so by Proposition \ref{pgensing} it is the cone on a homogeneous
quadratic hypersurface. By induction on the decomposition 
\w[,]{(\Gamma_{n},\Lambda_{n})_{n=m}^{M-1}} we may assume that at
the $n$-th stage the configuration \w{\Vp_{n}\in\Cs(\Gamma_{n})}
has a neighborhood $U$ of the stated form. By Definition
\ref{dgennt} we know that the work map
\w{\phi_{n}:\Cs(\Lambda_{n})\to\RR{d}} is a submersion at
\w[,]{\Vk_{n}} so it is work-equivalent at \w{\Vk_{n}} (Definition
\ref{dlocew}) to a projection \w{\pi:\RR{N_{n}}\to\RR{d}} (see
\cite[Theorem 7.8]{LeeS}). Therefore, in the pullback
\w{\Cs(\Gamma_{N+1})} the configuration
\w{\Vp_{n+1}=(\Vp_{n},\Vk_{n})} has a neighborhood
\w{U\times\RR{N_{n}-d}} \wh which is again of the required form.
\end{proof}

\begin{remark}\label{rmechanism}
Note that if \w{\Gamma=\Gamma_{M}} has a decomposition
\w{(\Gamma_{n},\Lambda_{n})_{n=m}^{M-1}} as in \S \ref{rip} and
\wref{etransv} holds at \w{\Vp_{n+1}=(\Vp_{n},\Vk_{n})} for each
$n$, then the configuration \w{\Vc=\Vp_{M}\in\Cs(\Gamma_{n})=\CsG}
is smooth, of course. Thus we obtain a mechanical interpretation
of \emph{all} differentiable singularities in any \cspace: namely,
they must occur at a kinematic singularity of type I for some
sub-mechanism \w{\Gamma_{n}} of $\Gamma$ \wh that is, a (smooth)
configuration \w{\Vp_{n}\in\Cs(\Gamma_{n})} at which the work map
\w{\psi_{n}:\Cs(\Gamma_{n})\to\RR{d}} is not a submersion (see
\cite{GAngS}).

In fact, more than this is required, since at the same point $\Vc$
another sub-mechanism \wh namely, the open chain \w{\Lambda_{n}}
\wh must be aligned, and it must be ``co-aligned'' with
\w{\Vp_{n}} in the sense that together they are $S$-equivalent to
an aligned closed chain (see proof of Proposition \ref{pgensing}).
We call this situation a \emph{conjunction} of two kinematic
singularities.
\end{remark}

%
%
\section{Example: a triangular planar linkage}
\label{ctriang}

We now consider an explicit example, which shows how all singular
configurations of a certain type of planar linkage can be
identified, by making use of a non-trivial $S$-equivalence.

\begin{mysubsection}{Parallel polygonal linkages}
\label{sppm}
In \cite{SSBlaCP}, a certain class of mechanisms were studied, called
\emph{parallel polygonal linkages}. These consist of two polygonal
\emph{platforms}. The first is the \emph{fixed} platform, which is
equivalent to fixing in \w{\RR{d}} the initial point \w{\xjk{i}{0}}
of each of $k$ open chains (called \emph{branches})
\wb[,]{1\leq i\leq k} of lengths \w[,]{\nj{1},\dotsc,\nj{k}} respectively.
The terminal point \w{\xjk{i}{\nj{i}}} of the $i$-th branch is
attached to the $i$-th vertex of a rigid planar $k$-polygon $\Pc$,
called the \emph{moving} platform.
See Figure \ref{fpentagon}.

\begin{figure}[htb]
\epsfysize=3cm
\leavevmode \epsffile{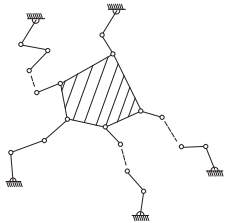}
\caption{A pentagonal planar mechanism}
\label{fpentagon}
\end{figure}

In the planar case, it was shown in \cite[Proposition 2.4]{SSBlaCP}
that a \emph{necessary} condition for a configuration $\Vc$ of such a
linkage $\Gamma$ to be singular is that one of the following holds:
\begin{enumerate}
\renewcommand{\labelenumi}{(\alph{enumi})~}
\item Two of its branch configurations \w{\Vj{i_{1}}} and
\w{\Vj{i_{2}}} are aligned, with coinciding direction lines
\w[.]{\Line(\xjk{i_{1}}{0},\xjk{i_{1}}{\nj{i_{1}}})=
\Line(\xjk{i_{2}}{0}\xjk{i_{2}}{\nj{i_{2}}})}
\item Three of its branch configurations are aligned, with direction lines
in the same plane meeting in a single point $P$ (see Figure \ref{ftriangle}).
\end{enumerate}

\begin{figure}[htb]
\epsfysize=5cm
\leavevmode \epsffile{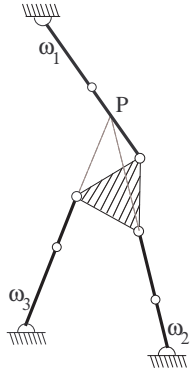}
\caption{Singular configuration of type (b)}
\label{ftriangle}
\end{figure}

For simplicity we assume that \w[,]{k=3} so the two platforms are
triangular. 
\end{mysubsection}

\begin{remark}\label{rtypea}
In the type (a) singularity there is obviously a
sub-mechanism \w{\Gp} which is isomorphic to an aligned closed
chain, so the corresponding configuration \w{\Vp} is singular.
Evidently, the caveat exemplified in \S \ref{egsing} does not
apply here, so in fact $\Vc$ is singular in \w[.]{\CsG}
\end{remark}

\begin{mysubsection}{A sub-mechanism and its equivalent open chain}
\label{sseoc}
We shall now show that the same holds (generically) for type (b),
using the approach of Section \ref{cpbc}.

Consider the sub-mechanism \w{\Gp} of $\Gamma$ obtained by
omitting the third branch (but retaining the fixed platform), with
base point at \w{\bx:=\xjk{3}{0}} (the fixed endpoint of the
omitted branch), and work point at \w{\by:=\xjk{3}{\nj{3}}} (the
moving endpoint of this branch). Let \w{\Vp} be the configuration
of \w{\Gp} corresponding to $\Vc$ of case (b) above (so in
particular the remaining two branches are aligned).

Assume that the first branch has \w{n:=\nj{1}} links, and the second
has \w{n':=\nj{2}} links. We may then choose ``internal''
parameters \w{(\phi_{1},\dotsc,\phi_{n})} for the first branch,
and \w{(\rho_{1},\dotsc,\rho_{n'})} for the second branch (as in
the proof of Proposition \ref{pzero}). We can then express the
lengths \w{\ell=\|\xjk{1}{0}\xjk{1}{n}\|} and
\w{m=\|\xjk{2}{0}\xjk{2}{n'}\|} as functions of
\w{(\phi_{1},\dotsc,\phi_{n})}  and
\w[,]{(\rho_{1},\dotsc,\rho_{n'})} respectively. Note that \w{\Gp}
has \w{n+n'+1} degrees of freedom, so one additional parameter is
needed. Two obvious choices are one of the ``base angles''
\w{\phi=\angle(\xjk{1}{n}\xjk{1}{0}\xjk{2}{0})} or
\w{\rho=\angle(\xjk{2}{n'}\xjk{2}{0}\xjk{1}{0})} for the two
branches (see Figure \ref{ftrisub}).

\begin{figure}[htb]
\epsfysize=5cm \leavevmode \epsffile{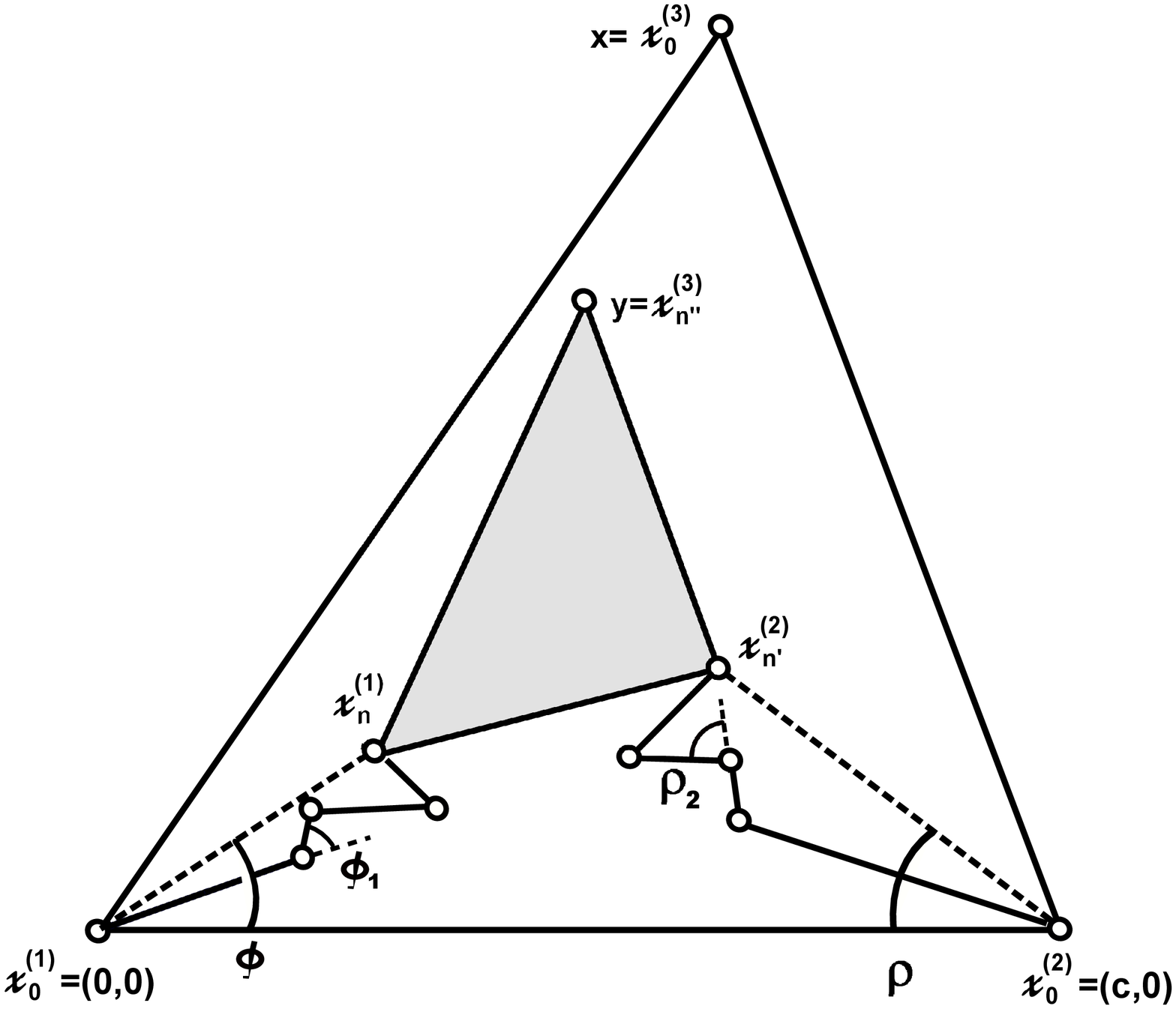} \caption{The
sub-mechanism \ $\Gp$} \label{ftrisub}
\end{figure}

However, for our purposes we shall need a different parameter, defined
as follows:

Let $\bz$ be the meeting point of the direction lines
\w{\Line(\xjk{1}{0}\xjk{1}{n})} and \w{\Line(\xjk{2}{0}\xjk{2}{n'})}
for the two branches (this is the point $P$ of Figure \ref{ftriangle}).
As our additional parameter we take the angle $\theta$ between the
direction line \w{\Line(\bx,\by)} for the (missing) third branch and
the line \w{\Line(\by,\bz)} (see Figure \ref{ftheta}). Note that
\w{\theta=0} or $\pi$ in our special configuration \w[.]{\Vp}
Letting \w[,]{N:=n+n'+1} the standard parametrization for the
open $N$-chain \w{\Gnl{N}} defines a (local) diffeomorphism
\w[.]{F:\Cs(\Gp)\to\Cs(\Gnl{N})}

\begin{figure}[htb]
\epsfysize=5cm \leavevmode \epsffile{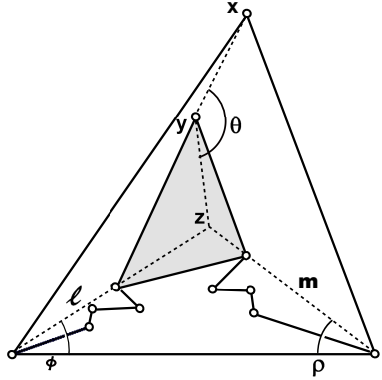} \caption{The
parameters $\theta$, $m$, and $\ell$ for the sub-mechanism \
$\Gp$} \label{ftheta}
\end{figure}

In order to show that $F$ is a work-equivalent at \w{\Vp} to
an aligned configuration \w{\Vj{N}} of \w{\Gnl{N}} (Definition \ref{dlocew}),
we must show that \w{\Vp} is a generic singularity for \w{\Gp} \wh
that is, that the reduced work map \w{\tpsi:\hCs(\Gp)\to\RR{}} has an
(isolated) singularity at \w[,]{\Vp} where $\tpsi$ assigns to any
configuration $\Vc$ of \w{\Gp} the length \w[.]{\tpsi(\Vc)=\|\bx,\by\|}

It is difficult to write $\tpsi$ explicitly as a function of
$\theta$: for this purpose it is simpler to use $\phi$ or $\rho$ as
above. However, if we fix the lengths
\w{\ell=\ell(\phi_{1},\dotsc,\phi_{n})} and
\w{m=m(\rho_{1},\dotsc,\rho_{n'})} of the direction vectors for the
two branches, the resulting linkage \w{\tGp} is a planar closed
$4$-chain with one degree of freedom (parameterized by $\phi$, say),
and the third vertex $\by$ of the moving triangle traces out a curve
in \w[,]{\RR{2}} called the \emph{coupler curve} for \w{\tGp} (cf.\
\cite[Ch.\ 4]{HallK}). Therefore, the infinitesimal effect of a change
in $\phi$ is the rotation of $\by$ about the point $\bz$ described
above (called the instantaneous point of rotation for \w[).]{\tGp} In
particular, the angle $\theta$ also changes, so we deduce that
\w{\ddp\theta/\ddp\phi\neq 0} at the aligned configuration
\w[.]{\Vp} This allows us to investigate the vanishing of
\w{\ddp\tpsi/\ddp\phi} instead of \w[.]{\ddp\tpsi/\ddp\theta}

This is the point where we are assuming genericity of \w[:]{\Vp} it
might happen that the coupler curve is singular precisely at this
point, in which case \w{\ddp\theta/\ddp\phi} may vanish, so we are no
longer guaranteed that $\theta$ is a suitable local parameter.  But
such instances of case (b) are not generic.

Since in the reduced \cspace\ \w{\hCs(\Gp)} we do not allow
rotation of \w{\Gp} about the base-point \w[,]{\xs=(x_{0},y_{0})}
we may assume that \w{\xjk{1}{0}=(0,0)} and
\w[.]{\xjk{2}{0}=(c,0)} Write \w{a:=\|\xjk{1}{n}\xjk{2}{n'}\|} and
\w{b:=\|\xjk{1}{n}\by\|} for the (fixed) sides of the moving
triangle (with fixed angle $\gamma$ between them), as in Figure
\ref{fparsub}.

\begin{figure}[htb]
\epsfysize=4cm \leavevmode \epsffile{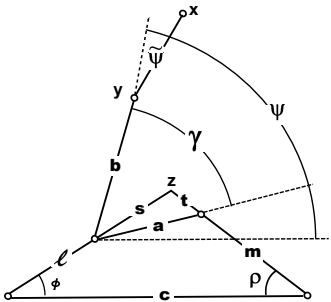}
\caption{Angles and lengths in the sub-mechanism \ $\Gp$}
\label{fparsub}
\end{figure}

We find that the following identities hold:
\begin{equation*}
\begin{split}
(\ell+s)\cos\phi~=&~c-(m+t)\cos\rho\\
(\ell+s)\sin\phi~=&~(m+t)\sin\rho\\
a\cos(\psi-\gamma)~=&~c-m\cos\rho-\ell\cos\phi\\
a\sin(\psi-\gamma)~=&~m\sin\rho-\ell\sin\phi
\end{split}
\end{equation*}
\noindent (where $\psi$ is the angle between side $b$ and the
$x$-axis), and:
\begin{equation*}
\begin{split}
a^{2}~=&~(c-m\cos\rho-\ell\cos\phi)^{2}+(m\sin\rho-\ell\sin\phi)^{2}\\
\tpsi^{2}~=&~(\ell\cos\phi+b\cos\psi-x_{0})^{2}+
   (\ell\sin\phi+b\sin\psi-y_{0})^{2}~.
\end{split}
\end{equation*}
\noindent After differentiating we find:
$$
\frac{\dd(\cos\rho)}{\ddp\phi}~=~\frac{\ell t}{ms}\sin\rho\hsp\text{and}\hsp
\frac{\dd(\cos\psi)}{\ddp\phi}~=~-\frac{\ell}{s}\sin\psi~,
$$
\noindent and we deduce that \w{\ddp\tpsi/\ddp\phi} vanishes if
and only if:
$$
b\ell\sin(\phi-\psi)+bs\sin(\psi-\phi)+(s\cos\phi,s\sin\phi)\cdot(-y_{0},x_{0})
+(b\cos\psi,b\sin\psi)\cdot(-y_{0},x_{0})=0.
$$
\noindent This formula expresses the fact that the area of the
triangle \w{\triangle \xjk{1}{0}\bz\bx} is the sum of the areas of
the quadrangle \w{\xjk{1}{0}\xjk{1}{n}\by\bx} and
\w[,]{\triangle\xjk{1}{n}\bz\by} which holds if and only if
\w{\xjk{1}{n}\by\bx} are aligned. From the formulas for
\w{\ell=\ell(\phi_{1},\dotsc,\phi_{n})} and
\w{m=m(\rho_{1},\dotsc,\rho_{n'})} we see that
\w{\ddp\tpsi/\ddp\phi_{i}} and \w{\ddp\tpsi/\ddp\rho_{j}} all
vanish at \w{\Vp} (as for any aligned open chain), so in case (b)
\w[,]{\nabla\tpsi=0} taken with respect to
\w[.]{(\theta,\phi_{1},\dotsc,\phi_{n},\rho_{1},\dotsc,\rho_{n'})}
Since all but one of the parameters are the standard internal
angles for open chains, we can check that the Morse indices for
the reduced work maps of \w{\Gp} and \w{\Gnl{N}} match up at
\w{\Vp} and \w[,]{F(\Vp)} showing that $F$ is indeed a
work-equivalence (see proof of Proposition \ref{pgensin}). 
Thus we may apply Proposition \ref{pgensing} to deduce that \w{\Vp} is a
cone singularity. 
\end{mysubsection}

\begin{summary}\label{sexample}
Since the caveat of \S \ref{egsing} does not apply to case (b), either 
(cf.\ \S \ref{rtypea}), we have shown that for a generic triangular
planar linkage $\Gamma$, any configuration \w{\Vc\in\CsG} satisfying
one of the \emph{necessary} conditions (a) and (b) of 
\cite[Proposition 2.4]{SSBlaCP} is ($S$-equivalent to) a generically
non-transverse configuration (Definition \ref{dgensing}). By 
Theorem \ref{tgensing} we can therefore deduce that it is indeed a
topological singularity \wh that is, conditions (a) and (b) are also 
\emph{sufficient}.

See \cite[Figure 8]{BBSShvaT} for an illustration of such a cone
singularity in a numerical example.
\end{summary}

%
%

\end{document}